\documentclass{llncs}
\usepackage{graphicx,amssymb,amsmath,llncsdoc}

\title{Stochastic Flips on Two-letter Words}

\author{Olivier {\sc Bodini}\inst{1} \and Thomas {\sc Fernique}\inst{2} \and Damien {\sc Regnault}\inst{2}}

\institute{
LIP6, CNRS \& Univ. Paris 6,\\
4 place Jussieu 75005 Paris - France,\\
\email{olivier.bodini@lip6.fr}
\and
LIF, CNRS \& Univ. de Provence\\
39 rue Joliot-Curie 13453 Marseille -- France\\
\email{\{thomas.fernique,damien.regnault\}@lif.univ-mrs.fr}
}

\begin{document}

\maketitle

\begin{abstract}
This paper introduces a simple Markov process inspired by the problem of quasicrystal growth.
It acts over two-letter words by randomly performing \emph{flips}, a local transformation which exchanges two consecutive different letters.
More precisely, only the flips which do not increase the number of pairs of consecutive identical letters are allowed.
Fixed-points of such a process thus perfectly alternate different letters.
We show that the expected number of flips to converge towards a fixed-point is bounded by $O(n^3)$ in the worst-case and by $O(n^{5/2}\ln{n})$ in the average-case, where $n$ denotes the length of the initial word.
\end{abstract}

\section*{Introduction}

Tilings are often used as a toy model for quasicrystals, with minimal energy tilings being characterized by local properties called \emph{matching rules}.
In this context, a challenging problem is to provide a theory for quasicrystals growth.
One of the proposed theories relies on a relaxation process (\cite{janot} p. 356): a tiling with many mismatches is progressively corrected by local transformations called \emph{flips}.
Ideally, the tiling eventually satisfies all the matching rules and thus shows a quasicrystalline structure.
It is compatible with experiments, where quasicrystals are obtained from a hot melt by a slow cooling during which flips really occur.
It is however unclear whether only flips can explain successful coolings or if other mechanisms should also be taken into account.
This question is deeply related with the convergence rate of such a flip-correcting process.\\

A relaxation process which aims to be physically realist is described in \cite{aperiodic}.
It considers general cut and project tilings of any dimension and codimension, and performs each flip which modifies by $\Delta E$ the energy of the tiling with a probability depending not only on $\Delta E$ but also on a temperature parameter $T$ such that the stationary distribution (at fixed $T$) is the Boltzmann distribution.
In this paper, we focus on a very rough version of this general process.
First, we consider only tilings of dimension and codimension one, which correspond to two-letter words.
Second, the flips with a corresponding $\Delta E$ being above a fixed threshold are equiprobably performed, while the other flips are simply forbidden.\\

The paper is organized as follows.
In Sec. \ref{sec:settings}, we clearly and formally state the definition of the stochastic process we consider, as well as the main question we are interested in, that is its \emph{convergence time}.
In Sec. \ref{sec:convergence1}, we state and prove the main result of this paper, which shows that the expected number of flips performed by the process to converge is cubic in the length of the initial word (Theorem \ref{th:convergence1}).
The proof mainly relies on a well-chosen function, called \emph{variant}, whose expected value strictly decreases.
This result is then extended in Sec. \ref{sec:convergence2} to the case where the initial words are randomly chosen according to some particular distribution which aims to be physically realist (Theorem \ref{th:convergence2}).
We conclude the paper by a short section discussing perspectives.

\section{Settings}
\label{sec:settings}

A \emph{configuration of length $n$} is a word $w=w_1\ldots w_n$ over $\{1,2\}$ with $|w|_1=|w|_2$, \emph{i.e.}, with as many occurrences of the letter $1$ as of the letter $2$.
Such an object is also sometimes called a \emph{grand Dyck path}, or a \emph{bridge}.
We denote by $\mathcal{W}_n$ the set of configurations of length $n$.\\

It is convenient to represent a configuration $w$ as the broken line of the Euclidean plane linking the points $(k,|w_1\cdots w_k|_1-|w_1\cdots w_k|_2)_{k=0\ldots,|w|}$ (see Fig. \ref{fig:config}).\\

\begin{figure}[hbtp]
\centering
\includegraphics[width=0.8\textwidth]{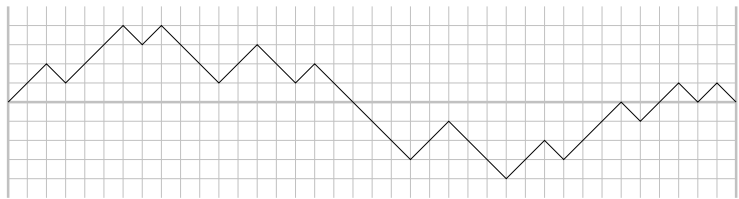}
\caption{The configuration $11211121222112212222211222112111211212$.}
\label{fig:config}
\end{figure}

A \emph{flip} at position $1\leq i<n$ on a configuration $w\in\mathcal{W}_n$ is the local transformation which exchanges $w_i$ and $w_{i+1}$, provided that these letters are different.
The \emph{height} of such a flip is the integer $|w_1\cdots w_{i-1}|_1-|w_1\cdots w_{i-1}|_2$.
Geo\-me\-tri\-cally, performing a flip on a configuration corresponds to move upwards or downwards by $2$ a point of the broken line which represents the configuration.\\

Flip thus acts over configurations, and it is not hard to check that any two configurations with the same length are connected by a sequence of flips.\\

\begin{figure}[hbtp]
\centering
\includegraphics[width=\textwidth]{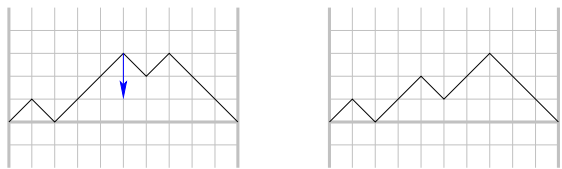}
\caption{A flip of height $2$ performed at position $5$ (leftwards). On the representation, the height and the position of a flip are the abscissa and the ordinate of the center between the two positions of the moved vertex (that is, the middle of the arrow).}
\label{fig:flip}
\end{figure}

A configuration $w$ is said to have a \emph{mismatch} at position $i$ if $w_i=w_{i+1}$.
For example, the configuration on Fig. \ref{fig:config} has $18$ mismatches, while configurations on Fig. \ref{fig:flip} both have $4$ mismatches.
The total number of mismatches of $w$ is denoted by $E(w)$.
One easily checks that it ranges from $0$ for configurations $(12)^n$ and $(21)^n$ to $2(n-1)$ for configurations $1^n2^n$ and $2^n1^n$.
Physically, $E(w)$ can be though as the \emph{energy} of the configuration $w$, with the configurations $(12)^n$ and $(21)^n$ thus being ground states.\\

We now define a Markov chain on $\mathcal{W}_n$ that we call \emph{cooling process}.
It starts from $w_0\in\mathcal{W}_n$ and produces a sequence $(w_t)_{t=1,2,\ldots}$ defined by:
$$
w_{t+1}=\left\{\begin{array}{ll}
w_t & \textrm{if }E(w_t)=0,\\
w' & \textrm{otherwise},
\end{array}\right.
$$
where $w'$ is obtained by performing on $w_t$ a flip uniformly chosen among the flips which do not increase the number of mismatches\footnote{In other words, the $\Delta E$ threshold discussed into the introduction is equal to zero.}.
Note that a flip modifies the number of mismatches by at most $2$.
Flips which decrease the number of mismatches are said to be \emph{irreversible} (they cannot be performed back) while the ones which do not modify the number of mismatches are said to be \emph{reversible}.

\begin{figure}[hbtp]
\centering
\includegraphics[width=0.8\textwidth]{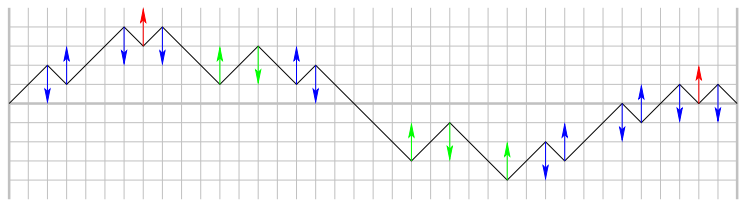}
\caption{Flips can increase the number of mismatches (red arrows), let it unchanged (blue arrows) or decrease it (green arrows).
The two last are respectively called reversible and irreversible flips: they are the only ones performed during the cooling process.}
\label{fig:process}
\end{figure}

Since the number $E(w_t)$ of mismatches is non-increasing, it is natural to ask whether it reaches zero or not, and at which rate.
We therefore introduce the \emph{convergence time} of this process, which is the random variable $T$ counting the number of steps required to transform the initial configuration $w_0$ into one of the two stable configurations without any mismatch, that is
$$
T:=min\{t\geq 0~|~E(w_t)=0\}.
$$
Then, the \emph{worst expected convergence time} is defined by
$$
\widehat{T}(n):=\max_{w\in \mathcal{W}_n} \mathbb{E}(T|w_0=w).
$$
We are not only interested in the worst case, but also in the average case.
Given a distribution $\mu$ over $\mathcal{W}_n$, we define the \emph{$\mu$-averaged expected convergence time} by
$$
T_\mu(n):=\sum_{w\in\mathcal{W}_n} \mu(w)\mathbb{E}(T|w_0=w).
$$

\section{Bounding the worst expected convergence time}
\label{sec:convergence1}

In this section, we prove the main result of this paper:

\begin{theorem}\label{th:convergence1}
The worst expected convergence time is cubic:
$$
\widehat{T}(n)=O(n^3).
$$
\end{theorem}

We will rely on the following proposition (proven, \emph{e.g.}, in \cite{fmst}) which bounds from above the expected convergence time of so-called \emph{variants}:

\begin{proposition}\label{prop:fmst}
Consider a Markov chain $(x_t)_{t\geq 0}$ over a space $\Omega$ and a positive real map $\phi:\Omega\to[a,b]$, called \emph{variant}.
Assume that there is $\varepsilon>0$ such that whenever $\phi(x_t)>a$, $\mathbb{E}[\phi(x_{t+1})-\phi(x_t)|x_t]\leq -\varepsilon$.
Then, the expected value of the random variable $T:=\min\{t|\phi(x_t)= a\}$ satisfies
$$
\mathbb{E}(T)\leq \frac{\mathbb{E}(\phi(x_0))}{\varepsilon}.
$$
\end{proposition}

Variants are thus sort of ``strict'' supermartingales, and we need to define a suitable one in order to bound the expected convergence time towards mismatch-free configurations.
Since the number $E$ of mismatches is non-increasing by definition of the cooling process, we are naturally tempted to use it as a variant.
However, there are configurations having only reversible flips, whose expected variation of mismatches is thus equal to zero (as, for example, $w=1211212212$).
In order to refine, we introduce the notion of \emph{Dyck factors} (see Fig. \ref{fig:dyck_factor}):

\begin{definition}
Consider a configuration $w=p\cdotp v\cdotp s$.
The factor $v=v_1\cdots v_k$ is said to be a \emph{positive Dyck factor} of \emph{length} $k$ and \emph{height} $|p|_1-|p|_2$ if:
$$
|v|_1=|v|_2,\qquad
\forall i\in\{1,\ldots,k\},~|v_1\cdots v_i|_1\geq |v_1\cdots v_i|_2,\qquad
|p|_1-|p|_2\geq 0.
$$
Exchanging letters $1$ and $2$ in the above conditions defines a \emph{negative Dyck factor}.
A Dyck factor is \emph{maximal} if no Dyck factor with the same height contains it.
\end{definition}

\begin{figure}[hbtp]
\centering
\includegraphics[width=0.8\textwidth]{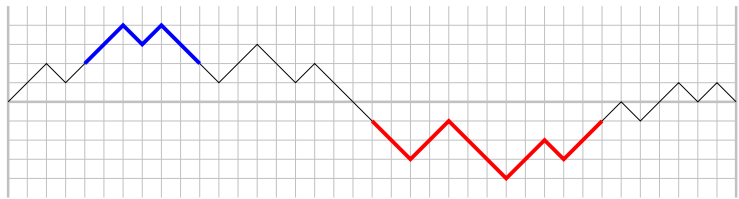}
\caption{A positive maximal Dyck factor of height $2$ and length $6$ (in blue), and a negative maximal Dyck factor of height $-1$ and length $12$ (in red).
}
\label{fig:dyck_factor}
\end{figure}

Let us stress that Dyck factors are either positive or negative, \emph{i.e.}, their representations do not cross the $y=0$ axis.
In particular, any configuration can be decomposed in its maximal Dyck factors of height zero, with signs alternating.\\

\noindent We use Dyck factors to define our variant, which depends on a parameter $\alpha$:

\begin{definition}\label{def:variant}
For $\alpha\in(0,1)$, let $\phi_\alpha$ be defined on a configuration $w$ by:
$$
\phi_\alpha(w)=\sum_{v\in \textrm{DF}(w)} (1+|v|_1)^\alpha,
$$
where $\textrm{DF}(w)$ denotes the set of maximal Dyck factors of $w$.
\end{definition}

Contrary to the number $E$ of mismatches, $\phi_\alpha$ can be increased by some flips.
However, the concavity of $x\to x^\alpha$, which gives to small Dyck factors a weight proportionally bigger than the weight given to long Dyck factors, will ensure that the average variation over all the flips of a configuration is always negative.
More precisely:

\begin{lemma}\label{lem:decreasing_variant}
Let $w\in\mathcal{W}_{2n}$.
If $E(w_t)>0$, then the variant $\phi_\alpha$ satisfies
$$
\mathbb{E}(\phi_\alpha(w_{t+1})-\phi_\alpha(w_t)|w_t)\leq -\frac{\alpha(1-\alpha)}{2}n^{\alpha-2}.
$$
\end{lemma}

\begin{proof}
Consider a configuration $w=w_1\cdots w_{2n}$, with $E(w)>0$.
This ensures that at least one flip (reversible or irreversible) can be performed on $w$.\\

Let us first assume that only reversible flips can be performed on $w$, that is, neither $1122$ nor $2211$ are factors of $w$.
This is equivalent to say that maximal Dyck factors contain at least four letters.
We will group these reversible flips by pairs and prove that the average variation of $\phi_\alpha$ over each pair is bounded from above by $\frac{\alpha(1-\alpha)}{2}n^{\alpha-2}$.\\
Each flip of positive height which transforms $21$ into $12$ (it increases $\phi_\alpha$) is between two positive flips which transform $12$ into $21$ (they decrease $\phi_\alpha$).
Reversibility ensures that one of these two flips has the same high as the central flip, while the other one is higher; we group the central flip and its higher neighbor.
We proceed symetrically for flips of negative height.
Fig. \ref{fig:pairing} illustrates this.\\

\begin{figure}[hbtp]
\centering
\includegraphics[width=0.8\textwidth]{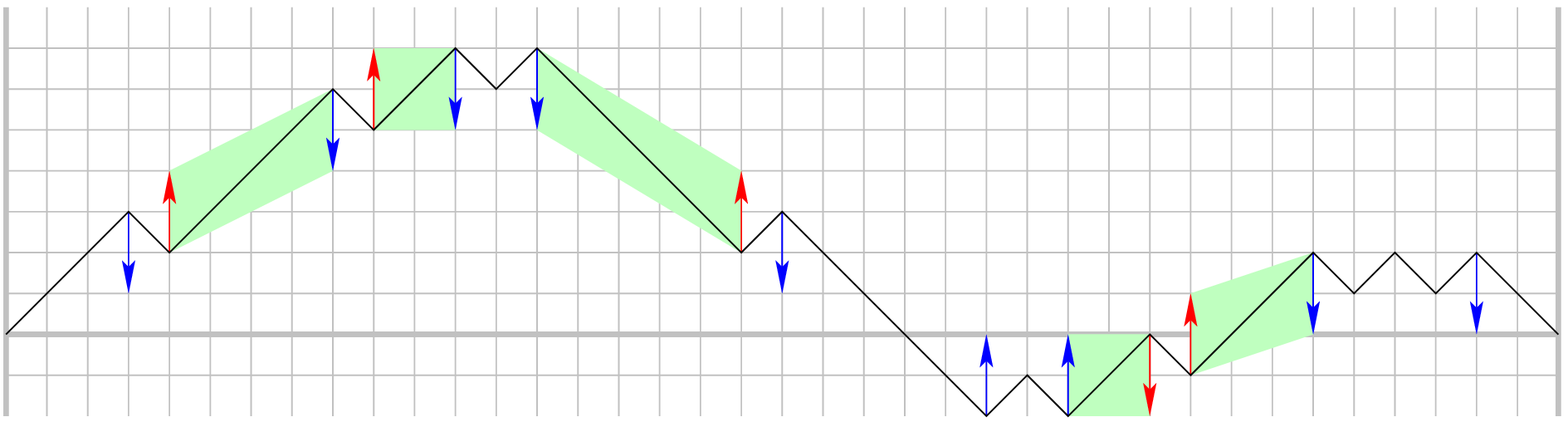}
\caption{
Each flip which increases $\phi_\alpha$ (red arrows) can be paired (light green boxes) with a flip which decreases $\phi_\alpha$ (blue arrows), such that the average variation of $\phi_\alpha$ over the two flips is negative.
Isolated flips can only decrease $\phi_\alpha$.}
\label{fig:pairing}
\end{figure}

Consider two paired reversible flips, say of positive height (the negative height case is symmetric).
Performing the lowest flip increases by $1$ the number $p$ of $1$'s in the maximal positive Dyck factor of $w$ starting at the first letter between the pair of flips, while performing the highest flip decreases by $1$ the number $q$ of $1$'s in some other maximal positive Dyck factor of $w$; one moreover has $p\geq q$ since the latter Dyck factor is higher than the former one (see Fig. \ref{fig:delta_pair}).

\begin{figure}[hbtp]
\centering
\includegraphics[width=0.8\textwidth]{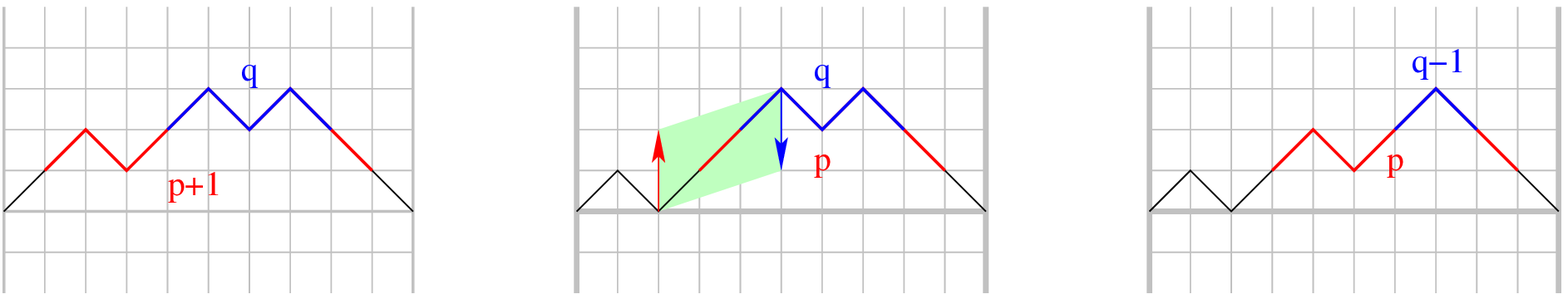}
\caption{
Consider two paired flips (central configuration): the maximal Dyck factor extended by the flip increasing $\phi_\alpha$ (red arrow, performed leftwards) is greater than the one shortened by the flip decreasing $\phi_\alpha$ (blue arrow, performed rightwards), \emph{i.e.}, $p\geq q$.
}
\label{fig:delta_pair}
\end{figure}

\noindent Thus, the average variation of $\phi_\alpha$ is
$$
\Delta:=\frac{1}{2}\left((p+2)^\alpha-(p+1)^\alpha+q^\alpha-(q+1)^\alpha\right).
$$
Since $x\to x^\alpha$ is concave for $\alpha\in (0,1)$, $p\geq q$ yields
$$
\Delta\leq \Delta':=\frac{1}{2}\left((p+2)^\alpha-(p+1)^\alpha+p^\alpha-(p+1)^\alpha\right).
$$
For $k\geq 1$, let $u_k=\frac{\alpha(\alpha-1)\ldots(\alpha-k+1)}{k!}$.
Note that $u_{2k}< 0$.
One computes:
\begin{eqnarray*}
\frac{2\Delta'}{(p+1)^\alpha}
&=&\left(1+\frac{1}{p+1}\right)^\alpha+\left(1-\frac{1}{p+1}\right)^\alpha-2\\
&=&1+\sum_{k\geq 1}\frac{u_k}{(p+1)^k}+1+\sum_{k\geq 1}(-1)^k\frac{u_k}{(p+1)^k}-2\\
&=&\sum_{k\geq 1}2\frac{u_{2k}}{(p+1)^{2k}}
\leq 2\frac{u_2}{(p+1)^2}
=\frac{\alpha(\alpha-1)}{(p+1)^2}.
\end{eqnarray*}
With $p<n$, this yields the claimed bound for the pair:
$$
\Delta\leq \Delta'
\leq \frac{(p+1)^\alpha}{2}\times\frac{\alpha(\alpha-1)}{(p+1)^2}
= -\frac{\alpha(1-\alpha)}{2}(p+1)^{\alpha-1}\leq -\frac{\alpha(1-\alpha)}{2}n^{\alpha-2}.
$$
Last, note that there are some flips which have not been paired, namely the ones which are immediatly after or before a crossing of the $y=0$ axis.
They however do not cause trouble, because each of them decreases a maximal Dyck factor of length at most $n$, hence decreases $\phi_\alpha$ by at least
$$
(n-1)^\alpha-n^\alpha\leq -\alpha n^{\alpha-1}\leq  -\frac{\alpha(1-\alpha)}{2}n^{\alpha-2}.
$$
Let us now consider the general case, \emph{i.e.}, when there are also irreversible flips.
We will show that an irreversible flip increases $\phi_\alpha$ lesser than (resp. decreases $\phi_\alpha$ more than) two reversible flips do.
Hence, the average variation of $\phi_\alpha$ can only be smaller than in the previous case (which is already negative).
Intuitively, it is as if each irreversible flip is splitted into two reversible flips before performing the above pairing process.\\
Let us be more precise.
Consider, first, the case of a positive flip which transforms $21$ into $12$.
It replaces two maximal Dyck factors, say with respectively $p$ and $q$ letters $1$, by one with $p+q+1$ letters $1$ (see Fig. \ref{fig:groupe_flips_2}).
Thus, on the one hand:
$$
\Delta=(p+q+2)^\alpha-(p+1)^\alpha-(q+1)^\alpha.
$$
On the other hand, the variation that would result of two reversible flips is:
$$
\Delta'=(p+2)^\alpha-(p+1)^\alpha+(q+2)^\alpha-(q+1)^\alpha.
$$
One thus has $\Delta\leq\Delta'$, since for $0<\alpha<1$:
$$
(p+2)^\alpha+(q+2)^\alpha\geq((p+2)+(q+2))^\alpha=(p+q+4)^\alpha\geq(p+q+2)^\alpha.
$$

\begin{figure}[hbtp]
\centering
\includegraphics[width=0.6\textwidth]{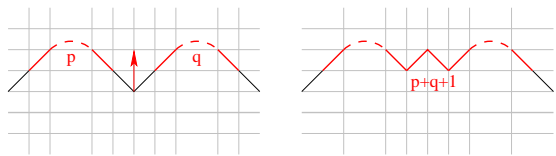}
\caption{An irreversible flip which replaces two maximal Dyck factors with respectively $p$ and $q$ letters $1$ by a single maximal Dyck factor with $p+q+1$ letters $1$.}
\label{fig:groupe_flips_2}
\end{figure}

We then consider the case of a positive irreversible flip which transforms $12$ into $21$.
This just removes a maximal Dyck factor with one letter $1$ (see Fig. \ref{fig:groupe_flips_3}).
Thus, on the one hand:
$$
\Delta=-2^\alpha.
$$

\begin{figure}[hbtp]
\centering
\includegraphics[width=0.6\textwidth]{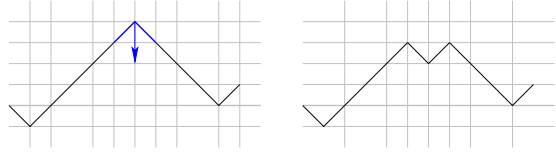}
\caption{An irreversible flip which removes a maximal Dyck factor with one letter $1$.}
\label{fig:groupe_flips_3}
\end{figure}

On the other hand, each of the two neighbor flips of this flip (one on the left, the other on the right, both transforming $21$ into $12$) either decreases the variant if it is negative, or increases it by:
$$
\Delta'=(p+2)^\alpha-(p+1)^\alpha\leq 3^\alpha-2^\alpha.
$$
The average variation of $\phi_\alpha$ is thus bounded from above by:
$$
\frac{2\Delta'+\Delta}{3}\leq \frac{2(3^\alpha-2^\alpha)-2^\alpha}{3}.
$$
A computation shows that, for $n\geq 2$, this quantity is lesser than $-\frac{\alpha(1-\alpha)}{2}n^{\alpha-2}$ for any $\alpha\in(0,1)$, with equality for $\alpha=1$.
The case of a negative irreversible flip is symmetric.
This concludes the proof.\hfill\qed
\end{proof}

This lemma thus provides a suitable $\varepsilon$ (depending on $\alpha$ and $n$) for Prop. \ref{prop:fmst}.
Let us now bound $\phi_\alpha$ and show that it is minimal for mismatch-free configurations:

\begin{lemma}\label{lem:initial_bound}
For any configuration $w\in\mathcal{W}_{2n}$, one has
$$
(1+n)^\alpha\leq \phi_\alpha(w)\leq (2n)^{\alpha+1},
$$
with the lower bound beeing reached if and only if $E(w)=0$.
\end{lemma}

\begin{proof}
Let us first focus on the upper bound.
We prove by induction on $|v|$ that, for any positive configuration $v$,
$$
\phi_\alpha(v)\leq (1+|v|_1)^{\alpha+1}.
$$
The inequality is satisfied for $n=2$ since one has $\phi_\alpha(12)=2^\alpha$.
Let us now consider a configuration $v$ of length $n$, and assume that the result holds for any shorter positive configuration.
Let $u_1,\ldots,u_p$ be the maximal (positive) Dyck factors of $v$ of height $1$.
One thus has
$$
\phi_\alpha(v)
=(1+|v|_1)^\alpha+\sum_{i=1}^p\phi_\alpha(u_i)
\leq (1+|v|_1)^\alpha+\sum_{i=1}^p(1+|u_i|_1)^{\alpha+1}.
$$
We then use the convexity of $x\to x^{\alpha+1}$ and compute
$$
\phi_\alpha(v)
\leq (1+|v|_1)^\alpha+\left(\sum_{i=1}^p 1+|u_i|_1\right)^{\alpha+1}
=(1+|v|_1)^\alpha+|v|_1^{\alpha+1}
\leq (1+|v|_1)^{\alpha+1}.
$$
The claimed result follows by induction.
It also holds for negative configurations.\\
Let us now consider the case of a general configuration $w$.
Let $v_1,\ldots,v_q$ be its maximal Dyck factors of height $0$.
Since they are alternatively positive and negative configurations, one relies one the previous result to compute
$$
\phi_\alpha(w)
=\sum_{i=1}^q\phi_\alpha(v_i)
\leq\sum_{i=1}^q(1+|v_i|_1)^{\alpha+1}
\leq\left(\sum_{i=1}^q 1+|v_i|_1\right)^{\alpha+1}
=(q+|w|_1)^{\alpha+1}.
$$
Since $q\leq|w|_2$, this yields $\phi_\alpha(w)\leq |w|^{\alpha+1}=(2n)^{\alpha+1}$, hence the upper bound.\\

Let us now focus on the lower bound.
Let $w$ be a configuration and $v_1,\ldots,v_q$ be its maximal Dyck factors of height $0$.
By relying on the concavity of $x\to x^\alpha$, one computes
$$
\phi_\alpha(w)
=\sum_{i=1}^q\phi_\alpha(v_i)
=\sum_{i=1}^q(1+|v|_1)^\alpha
\geq \left(\sum_{i=1}^q1+|v|_1\right)^\alpha
=(q+|w|_1)^\alpha.
$$
Since $|w|_1=n$ and $q\geq 1$, this proves the claimed lower bound.
Moreover, if this bound is reached, then $q=1$, that is, $w$ is either positive or negative.
In both cases, $w$ has a maximal Dyck factor with $n$ letters $1$, which has weight $(1+n)^\alpha$.
This ensures that there is no other maximal Dyck factors.
Hence, $w$ is one of the configuration $(12)^n$ or $(21)^n$, which are exactly the mismatch-free ones and are both mapped onto $(1+n)^\alpha$ by $\phi_\alpha$.
This proves the claimed result.\hfill\qed
\end{proof}

\noindent This bound is optimal up to a factor $2^{\alpha+1}$ since one computes $\phi_\alpha(1^n2^n)=n^{\alpha+1}$.\\

\noindent Th. \ref{th:convergence1} then follows from Prop. \ref{prop:fmst} with $\varepsilon=\frac{\alpha(1-\alpha)}{2}n^{\alpha-2}$ and $f_0=\phi_\alpha(w_0)\leq n^{\alpha+1}$:
$$
\mathbb{E}(\widehat{T}(n))\leq n^{\alpha+1}\times \frac{2n^{2-\alpha}}{\alpha(1-\alpha)}=\frac{2}{\alpha(1-\alpha)}n^3.
$$
Here, the best value for $\alpha$ is $\frac{1}{2}$.
Other values are useful only in the next section.

\section{Bounding the weighted expected convergence time}
\label{sec:convergence2}

Given a distribution $\mu$ over $\mathcal{W}_n$, one deduces from Prop. \ref{prop:fmst} a bound on the $\mu$-weighted expected convergence time, with $\varepsilon=\frac{\alpha(1-\alpha)}{2}n^{\alpha-2}$ (Lem. \ref{lem:decreasing_variant}):
$$
T_\mu(n)\leq\frac{1}{\varepsilon}\sum_{w\in\mathcal{W}_n}\mu(w)\phi_\alpha(w).
$$
Thus, to bound $T_\mu(n)$, we just need to compute the $\mu$-weighted value of $\phi_\alpha$.
It is however generally easier to compute the $\mu$-weighted value of the volume.
The volume $V(w)$ of a configuration $w$ is the area between its representation and the horizontal axis (for example, Fig. \ref{fig:flip}, both configurations have volume $15$).
One can then deduce bound on the $\mu$-weighted value of $\phi_\alpha$ from the simple following inequality, which holds for any $\alpha\in(0,1)$ and for any configuration $w$:
$$
\phi_\alpha(w)\leq V(w).
$$

For the uniform distribution $\upsilon$ over $\mathcal{W}_n$, the $\upsilon$-weighted value of the volume is known to be equivalent to $\frac{\sqrt{2\pi}}{8}n\sqrt{n}$ (see, \emph{e.g.}, \cite{flajolet}, p. 533).
This allows to prove:

\begin{theorem}\label{th:convergence2}
The uniformly-weighted expected convergence time $T_\upsilon$ satisfies
$$
T_\upsilon(n)=O(n^{5/2}\ln{n}).
$$
\end{theorem}

\begin{proof}
For any $\alpha\in(0,1)$, one has
$$
T_\upsilon(n)
\leq\frac{2n^{2-\alpha}}{\alpha(1-\alpha)}\times \sum_{w\in\mathcal{W}_n}\upsilon(w)V(w)
=\frac{\sqrt{2\pi}}{4\alpha(1-\alpha)}n^{7/2-\alpha}.
$$
One checks that this upper bound, seen as a function of $\alpha$, is minimal in $1-\frac{1}{\ln{n}}$.
For such an $\alpha$, one has $n^{7/2-\alpha}=n^{1/\ln{n}} n^{5/2}=\mathrm{e} n^{5/2}$.
This yields
$$
T_\upsilon(n)
\leq \frac{\sqrt{2\pi}}{4(1-\frac{1}{\ln{n}})\frac{1}{\ln{n}}}\mathrm{e} n^{5/2}
= \frac{\sqrt{2\pi} \mathrm{e}}{4(1-\frac{1}{\ln{n}})}n^{5/2}\ln{n}
=O(n^{5/2}\ln{n}).
$$
This shows the claimed result.\hfill\qed
\end{proof}

However, recall from the introduction that the context of this problem is the cooling of a quasicrystal from melt at very high temperature.
In particular, the natural distributions over initial configurations should be the one of the melt.
In the melt, the very high temperature imposes a very low threshold for  $\Delta E$, so that \emph{all} the flips are equiproblable, with no flip being any more forbidden.\\

Hence, we define the \emph{natural distribution} $\nu$ over $\mathcal{W}_n$ as the stationary distribution of the process which, at each step, performs uniformly at random one of the possible flips, without any restriction (contrarily to the cooling process described in Sec. \ref{sec:settings}).
It is a classical result of random walks on graph that this stationary distribution gives to each configuration a weight proportional to the number of flips which can be performed onto.
In particular, the configurations of maximal weight are $(12)^n$ and $(21)^n$, \emph{i.e.}, the stable configurations of the cooling process, while the configurations of minimal weight are $1^n2^n$ and $2^n1^n$, \emph{i.e.}, the worst configurations of the cooling process.
One can thus hope that the naturally-weighted expected convergence time $T_\nu$ is lower than the uniformly-weighted one.
But we only get a similar bound, because the naturally-weighted value of the volume has the same growth order as its uniformly-weighted value:

\begin{proposition}
The naturally-weighted value of the volume satisfies
$$
\sum_{w\in\mathcal{W}_n} \nu(w)V(w)\sim\frac{\sqrt{2\pi}}{8}n\sqrt{n}.
$$
\end{proposition}

\begin{proof}
We first introduce a combinatorial enumerative series $f$ describing only the positive configurations.
Let $a_{n,p,q}$ denotes the number of positive configurations of size $2n$, volume $q$ and onto which $p$ flips can be performed.
We set:
$$
f(z,u,v)=\sum_{(n,p,q)\in\mathbb{N}^3}a_{n,p,q}z^nu^pv^q.
$$
Let us find a functional equation for $f$.
Each positive configuration $w$ can be written $w=1u2v$, where both $u$ and $v$ are positive configurations, possibly empty.
The volume satisfies $V(w)=V(u)+|u|+1+V(v)$.
The number $F$ of flips satisfies either $F(w)=F(u)$ if $|v|=0$, or $F(w)=F(u)+1+F(v)$ otherwise.
This leads, using symbolic methods described in \cite{flajolet}, to the functional equation
\begin{equation}\label{eq:f}
f(z,u,v)=(zuv+zvf(zv,u,v))(1+uf(z,u,v)).
\end{equation}
We now introduce a similar combinatorial enumerative series $g$ describing all the configurations.
Each configuration $w$ can be written $w=v_1\cdots v_p$, where the $v_i$'s are alternating positive or negative configurations, hence described by $f$.
Once again, symbolic methods described in \cite{flajolet} thus lead to the functional equation
\begin{equation}\label{eq:g}
g(z,u,v)=\dfrac{2z(vf(zv,u,v)+uv)}{1-(zuv+zu^2v+zvf(zv,u,v)+zuv)}.
\end{equation}
The naturally-weighted value of the volume can then be derived from $g$ as follows:
$$
\sum_{w\in\mathcal{W}_{2n}} \nu(w)V(w)=\dfrac
{[z^n]{\frac {\partial ^{2}}{\partial v\partial u}}g\left(z,u,v\right)|_{u=1,v=1}}
{[z^n]{\frac {\partial }{\partial u}}g \left( z,u,v \right)|_{u=1,v=1}},
$$
where $[z^n]h$ denotes the coefficient of $z^n$ in the series $h$.
Let us compute this.\\

We first consider the denominator.
The differentiation of (\ref{eq:g}) with respect to $u$ yields for $\frac{\partial g}{\partial u}(z,u,v)$ a huge formula where appear $f(zv,u,v)$ and $\frac{\partial }{\partial u}f(zv,u,v)$.
On the one hand, $f(z,1,1)$ is simply the usual series describing Dyck paths:
$$
f(z,1,1)=\frac{1-\sqrt{1-4z}}{2z}.
$$
On the other hand, the differentiation of (\ref{eq:f}) with respect to $u$ yields
\begin{align*}
\frac{\partial f}{\partial u}(z,u,v)=
&\left(zv+vz\frac{\partial f}{\partial u}(zv,u,v)\right) \left(1+uf(z,u,v)\right)\\
&+ \left(zuv+zvf(zv,u,v)\right) \left(f(z,u,v)+u\frac{\partial f}{\partial u}(z,u,v)\right).
\end{align*}
Since $f(z,1,1)$ is known, this allows to compute
$$
\frac{\partial f}{\partial u}(z,u,v)|_{u=1,v=1}=\frac{\partial f}{\partial u}(zv,u,v)|_{u=1,v=1}=\frac{1-\sqrt{1-4z}-2z}{2z\sqrt{1-4z}}.
$$
Working this out in the equation for $\frac{\partial g}{\partial u}(z,u,v)$, one gets
$$
D(z):=\frac{\partial g}{\partial u}(z,u,v)|_{u=1,v=1}=\frac{2z}{(1-4z)^{3/2}}.
$$
For the numerator, a similar but even more tedious computation yields
$$
N(z):=\frac{\partial^2 g}{\partial u\partial v}(z,u,v)|_{u=1,v=1}=\frac{16z^3+4z^2+2z}{(1-4z)^3}.
$$
Both $D(z)$ and $N(z)$ are $\Delta$-analytic at $z=\frac{1}{4}$.
In this point, one computes
$$
D(z)\sim \frac{1}{2(1-4z)^{3/2}}
\qquad\textrm{and}\qquad
N(z)\sim\frac{1}{(1-4z)^3}.
$$
Asymptotics of the coefficient of $z^n$ in the series expansions of $D$ and $N$ can then be deduced from the following result (see \emph{e.g.}\cite{flajolet}, Chap. VI), which holds for $\alpha\notin\{-1,-2,\ldots\}$ and $\rho\in\mathbb{C}$:
$$
[z^n]\left(1-\frac{z}{\rho}\right)^{-\alpha}\sim \frac{1}{\Gamma(\alpha)}\rho^{-n}n^{\alpha-1}.
$$
Indeed, with $\rho$ equals to $\frac{1}{4}$ and $\alpha$ respectively equals to $\frac{3}{2}$ and $3$, this yields:
$$
[z^n]D(z)\sim\frac{1}{2\Gamma(3/2)}4^n\sqrt{n}
\qquad\textrm{and}\qquad
[z^n]N(z)\sim\frac{1}{\Gamma(3)}4^nn^2.
$$
One thus finally computes
$$
\sum_{w\in\mathcal{W}_{2n}} \nu(w)V(w)=\dfrac{[z^n]N(z)}{[z^n]D(z)}\sim\frac{2\Gamma(3/2)}{\Gamma(3)}n\sqrt{n}=\frac{\sqrt{\pi}}{2}n\sqrt{n}.
$$
The claimed result follows by replacing $2n$ by $n$ (the sum is taken over $\mathcal{W}_{n}$)
\hfill\qed
\end{proof}

\section{Perspectives}

Th. \ref{th:convergence1} and \ref{th:convergence2} provide \emph{upper} bounds on the expected convergence time of the cooling process, respectively in the worst case and in the average case (uniform and natural distributions).
Experiments suggest, especially in the worst case, that these bounds are tight (see Fig. \ref{fig:stats_pire_cas} and \ref{fig:stats_cas_moyen}).
This remains to be rigorously proven.\\

Moreover, this paper focuses only on the expectation of the convergence time, as the physically most significant moment.
It would be worth to consider also higher moments, in order to obtain the limit distribution of the convergence time.\\

Last but not least, two-letter words correspond to tilings of dimension one and codimension one.
But the physically interesting cases have dimension two or three and codimension at least two.
For example, the celebrated \emph{Penrose tiling} has dimension two and codimension three.
The analysis of these cases seems harder but rewarding: experiments show much faster convergence rates (see \cite{aperiodic}).


\begin{figure}[hbtp]
\centering
\includegraphics[width=\textwidth]{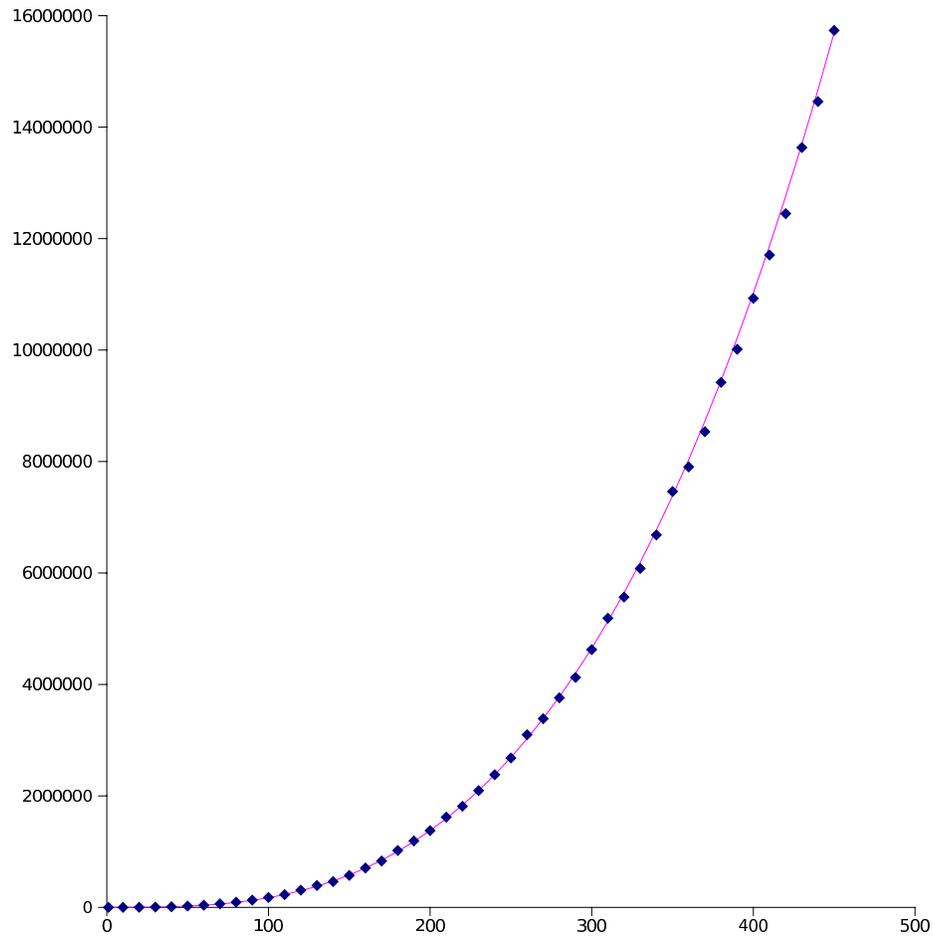}
\caption{
The blue points represent the average over $10$ simulations of the worst convergence time (on the $x$-axis, the length $n$ of configurations).
The pink curve $0.17n^3$ fits remarkably well with experimental values.
The standard deviation of the worst convergence time is experimentally very small, so that only $10$ simulations already provide a good approximation of its expectation.
}
\label{fig:stats_pire_cas}
\end{figure}

\begin{figure}[hbtp]
\centering
\includegraphics[width=\textwidth]{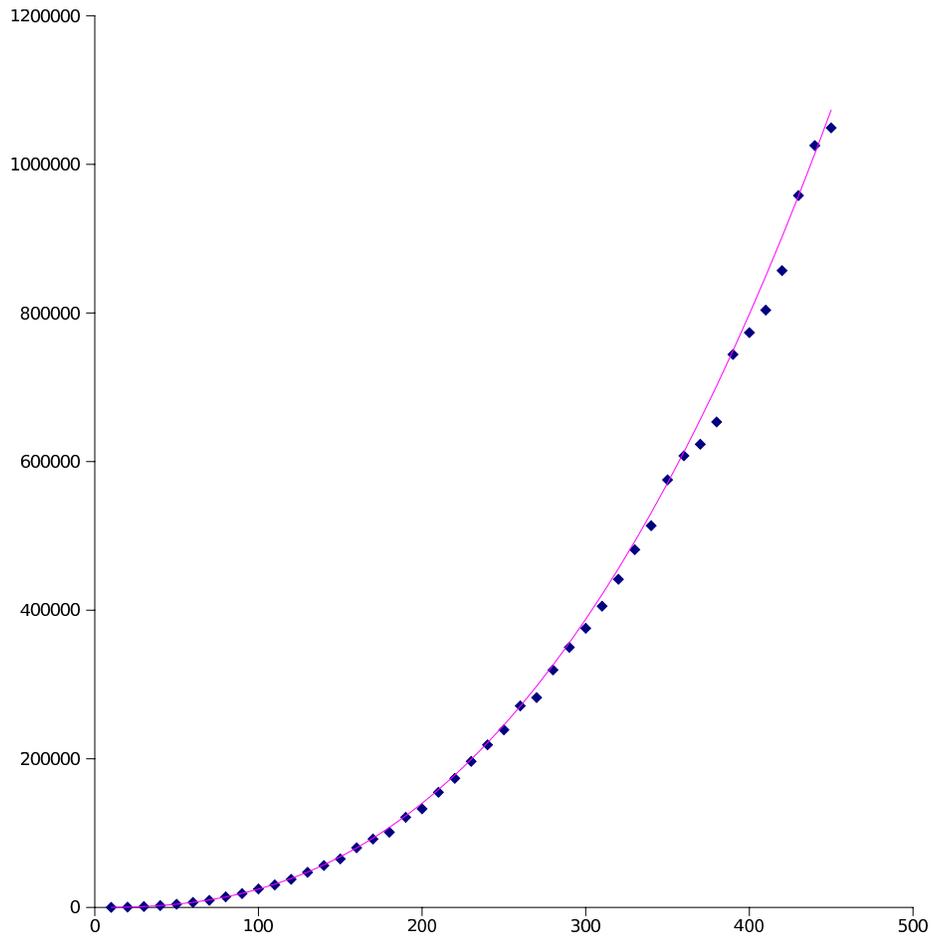}
\caption{
The blue points represent the average over $1000$ simulations of the uniformly-weighted convergence time (on the $x$-axis, the length $n$ of configurations).
The pink curve $0.24n^{5/2}\ln{n}$ fits rather well with experimental values.
The standard deviation of the uniformly-weighted convergence time is experimentally much greater than for the worst convergence time (indeed, each simulation starts from a configuration chosen uniformly at random), so that the number of simulations needed to get a good approximation is much greater too.
}
\label{fig:stats_cas_moyen}
\end{figure}


\end{document}